\documentclass[a4paper,11pt]{amsart}

\usepackage{amsmath, amsthm, amsfonts, amssymb}
\usepackage{enumerate}

\usepackage{color} 
\usepackage[colorlinks=true, linkcolor=blue, citecolor=magenta, menucolor=black]{hyperref}

\numberwithin{equation}{section}

\newtheorem*{main}{Main Theorem}
\newtheorem*{maincor}{Corollary}

\newtheorem{thm}{Theorem}[section]
\newtheorem{lem}[thm]{Lemma}

\newtheorem{prop}[thm]{Proposition}

\theoremstyle{definition}

\newtheorem{example}[thm]{Examples}

\newtheorem{df}[thm]{Definition}

\newtheorem{claim}[thm]{Claim}

\newcommand{\R}{\mathbf{R}}
\newcommand{\C}{\mathbf{C}}

\newcommand{\N}{\mathbf{N}}
\newcommand{\B}{\mathbf{B}}

\newcommand{\id}{\text{\rm id}}

\newcommand{\rL}{\mathord{\text{\rm L}}}
\newcommand{\rB}{\mathord{\text{\rm B}}}

\newcommand{\ri}{\mathord{\text{\rm i}}}

\newcommand{\Tr}{\mathord{\text{\rm Tr}}}

\newcommand{\ovt}{\mathbin{\overline{\otimes}}}

\newcommand{\dpr}{^{\prime\prime}}

\newcommand{\cH}{\mathcal{H}}

\begin{document}

\title[Connes' bicentralizer problem for $q$-deformed Araki--Woods algebras]{Connes' bicentralizer problem for $\bf q$-deformed Araki--Woods algebras}

\begin{abstract}
	Let $(H_\R, U_t)$ be any strongly continuous orthogonal representation of $\R$ on a real (separable) Hilbert space $H_\R$. For any $q\in (-1,1)$, we denote by $\Gamma_q(H_\R,U_t)\dpr$ the $q$-deformed Araki--Woods algebra introduced by Shlyakhtenko and Hiai. In this paper, we prove that $\Gamma_q(H_\R,U_t)\dpr$ has trivial bicentralizer if it is a type $\rm III_1$ factor. In particular, we obtain that $\Gamma_q(H_\R,U_t)\dpr$ always admits a maximal abelian subalgebra that is the range of a faithful normal conditional expectation. Moreover, using \'Sniady's work, we derive that $\Gamma_q(H_\R,U_t)\dpr$ is a full factor provided that the weakly mixing part of $(H_\R, U_t)$ is nonzero.
\end{abstract}

\author{Cyril Houdayer}
\address{Universit\'e Paris-Saclay \\ CNRS \\ Laboratoire de math\'ematiques d'Orsay \\ 91405 \\ Orsay \\ Institut Universitaire de France \\ FRANCE }
\address{Graduate School of Mathematical Sciences \\ The University of Tokyo \\ Komaba \\ Tokyo \\ 153-8914 \\ JAPAN}
\email{cyril.houdayer@universite-paris-saclay.fr}
\thanks{CH is supported by ERC Starting Grant GAN 637601, Institut Universitaire de France, and FY2019 JSPS Invitational Fellowship for Research in Japan (long term)}

\author{Yusuke Isono}
\address{RIMS, Kyoto University, 606-8502 Kyoto, JAPAN}
\email{isono@kurims.kyoto-u.ac.jp}
\thanks{YI is supported by JSPS KAKENHI Grant Number JP17K14201}


\maketitle

\section{Introduction}\label{Introduction}

Let $M$ be any $\sigma$-finite von Neumann algebra equipped with a faithful state $\varphi \in M_*$. The {\em asymptotic centralizer} $\mathrm{AC}(M,\varphi)$  and the {\em bicentralizer} $\mathrm{B}(M,\varphi)$ are defined by
\begin{align*}
	\mathrm{AC}(M, \varphi) &:= \left\lbrace (x_n)_n \in \ell^\infty(\N, M) \mid \lim_{n \to \infty} \|x_n \varphi - \varphi x_n\|_{M_*} = 0 \right\rbrace;\\
	\mathrm{B}(M, \varphi) &:=  \left\lbrace a \in M \mid a x_n - x_n a\to 0\ (n\to \infty) \text{ $\ast$-strongly for all }(x_n)_n \in \mathrm{AC}(M, \varphi) \right\rbrace .
\end{align*}
Here we used the notation $(a\varphi b)(x):=\varphi(bxa)$ for all $a,b,x\in M$. A. Connes conjectured that any type III$_1$ factor $M$ with separable predual has \textit{trivial bicentralizer}, namely $\mathrm{B}(M, \varphi)=\C$. (This triviality does not depend on the choice of $\varphi$.) 
U. Haagerup established a fundamental characterization of this condition and proved that any injective type III$_1$ factor has trivial bicentralizer \cite{Ha85}. This was the last piece of the classification of amenable factors and he obtained the uniqueness of amenable type III$_1$ factors.

	Although Connes' original problem remains open, most of concrete examples are known to have trivial bicentralizer. 
Indeed the following type $\rm III_1$ factors with separable predual have trivial bicentralizer:
\begin{itemize}
	\item Factors with almost periodic states \cite{Co74} (see also Lemma \ref{almost periodic lemma});

	\item Free Araki--Woods factors \cite{Ho08}, more generally any free product factors \cite{HU15};

	\item Semisolid factors \cite{HI15};

	\item Tensor product factors $P\ovt Q$, where $P$ is a type III$_1$ factor with trivial bicentralizer and $Q$ is any factor \cite{Ma18} (see also \cite[Lemma 7.5]{IM19}).

\end{itemize}
By \cite{Ha85, Po81}, the triviality of the bicentralizer algebra is equivalent to the existence of a maximal abelian subalgebra that is the range of a faithful normal conditional expectation (say \textit{with expectation} in short). Therefore all the aforementioned factors have maximal abelian subalgebras with expectation. 

	In this paper, we investigate the bicentralizer problem for the $q$-deformed Araki--Woods algebras introduced by Hiai \cite{Hi02}, where $q\in (-1,1)$. This is a $q$-analogue of Shlyakhtenko's free Araki--Woods factors \cite{Sh97} and hence it has a free product structure when $q=0$. When $q\neq 0$, however, it no longer has the free product structure and this is the reason why the bicentralizer problem is not solved for these algebras. 

	Recently it was observed in \cite[Proposition 3.3]{HI15} that $\mathrm B(M,\varphi)$ can be realized as a subalgebra in the ultraproduct von Neumann algebra: for any free ultrafilter $\omega $ on $\N$,
	$$ \mathrm B(M,\varphi) = (M^\omega_{\varphi^\omega})' \cap M \subset M^\omega.$$
See also \cite{AHHM18,Ma18} for recent developments in the bicentralizer problem, using ultraproduct techniques. 
We will use the ultraproduct framework to study $q$-Araki--Woods algebras. Our key observation is that for the $q$-Araki--Woods algebra $M_q = \Gamma_q(H_\R,U_t)\dpr$, the $q$-Araki--Woods algebra $\mathcal M_q = \Gamma_q((H_\R)_{U, \omega}, U_\omega)\dpr$ associated with the ultraproduct orthogonal representation sits as an intermediate subalgebra with expectation $M_q \subset \mathcal M_q \subset M_q^\omega$ (see Section \ref{section:equicontinuity} for further details). 

Our main theorem solves the bicentralizer problem for $q$-Araki--Woods algebras. Moreover, using \'Sniady's work \cite{Sn03}, we derive that $\Gamma_q(H_\R,U_t)\dpr$ is a full factor provided that the weakly mixing part of $(H_\R, U_t)$ is nonzero.

\begin{main}\label{thmA}
	Let $q\in (-1,1)$ and let $(H_\R, U_t)$ be a strongly continuous orthogonal representation such that the weakly mixing part is nonzero. Then $\Gamma_q(H_\R,U_t)\dpr$ is a full factor of type $\rm III_1$ which has  trivial  bicentralizer.
\end{main}

\begin{maincor}\label{corB}
	Any $q$-Araki--Woods algebra $\Gamma_q(H_\R,U_t)\dpr$ admits a maximal abelian subalgebra with expectation, provided that $H_\R$ is separable.
\end{maincor}

For any nonzero $\xi \in H_\R$, the abelian subalgebra $\{W_q(\xi)\}\dpr$ is called a \textit{generator subalgebra}. If $\xi$ is fixed by the action, then the generator subalgebra is a maximal abelian subalgebra (masa) with expectation \cite{Ri04,BM16}. However, if $\xi$ is not fixed by the action, then this subalgebra is not with expectation (and is not a masa if $\Gamma_q(H_\R,U_t)\dpr$ is diffuse \cite{BM19}). Hence if the action has no fixed points, probably the masa given in Corollary cannot be written in a concrete way.

\subsection*{Acknowledgments}
YI would like to thank Yuki Arano for useful comments and Mateusz Wasilewski for pointing out a gap of the proof in the first draft of the paper.

{
  \hypersetup{linkcolor=black}
  \tableofcontents
}

\section{Preliminaries}\label{Preliminaries}

Throughout the paper, all von Neumann algebras are assumed to be $\sigma$-finite.

\subsection*{Ultraproduct von Neumann algebras}
	Let $M$ be any $\sigma$-finite von Neumann algebra and we denote by $\ell^\infty(\N, M)$ the set of all norm bounded sequences. 
For any free ultrafilter $\omega$ on $\N$, define
\begin{align*}
	\mathcal I_{\omega} &= \left\{ (x_n)_{n} \in \ell^\infty(\N, M) \mid x_n \to 0 \text{ $\ast$-strongly as } n \to \omega \right\} ;\\
	\mathcal M^{\omega} &= \left \{ x \in \ell^\infty(\N, M) \mid  x \mathcal I_{\omega} \subset \mathcal I_{\omega} \text{ and } \mathcal I_{\omega}x \subset \mathcal I_{\omega}\right\}.
\end{align*}
The quotient C$^*$-algebra $M^\omega := \mathcal{M}^\omega/ \mathcal{I}_\omega$ is a von Neumann algebra, and we call it the \textit{ultraproduct von Neumann algebra} \cite{Oc85}. For any faithful state $\varphi\in M_*$, the assignment $\mathcal M^\omega \ni (x_n)_n \mapsto \lim_{n\to \omega} \varphi(x_n)$ induces a faithful normal state on $M^\omega$, which we write as $\varphi^\omega$. 
There is a natural embedding $M \subset M^\omega$ with expectation $E_M$ such that the embedding map is normal and $\varphi^\omega = \varphi \circ E_M$. For any $(x_n)_n\in \mathcal M^\omega$, we denote by $(x_n)_\omega$ its image in $M^\omega$. For more on ultraproduct von Neumann algebras, we refer the reader to \cite{Oc85,AH12}.

\subsection*{Almost periodic states}

	For any von Neumann algebra $M$ with a faithful state $\varphi\in M_*$, we will denote by $\Delta_\varphi$ the \textit{modular operator}, by $J_\varphi$  the \textit{modular conjugation}, and by $\sigma^\varphi_t$ $(t\in \R)$ the \textit{modular action}. The fixed point of $\sigma^\varphi$ is denoted by $M_\varphi \subset M$ and we call it the \textit{centralizer algebra}. We say that $\varphi$ is \textit{almost periodic} \cite{Co74} if the modular operator $\Delta_\varphi$ is diagonalizable.

The following lemma is well known. We include a short proof for the reader's convenience.

\begin{lem}\label{almost periodic lemma}
	Let $M$ be a von Neumann algebra equipped with an almost periodic state $\varphi\in M_*$. The following assertions hold true.
\begin{enumerate}
	\item[$\rm(1)$] If $M$ is diffuse, then $M_\varphi$ is also diffuse.
	\item[$\rm(2)$] There is a semifinite von Neumann subalgebra $N \subset M$ with expectation such that $N' \cap M \subset N$. In particular $M$ has a maximal abelian subalgebra with expectation if $M_*$ is separable.
\end{enumerate}
\end{lem}
\begin{proof}
	(1) If there is a nonzero projection $p\in M_\varphi$ which is minimal in $M_\varphi$, then $pMp$  is a von Neumann algebra with an almost periodic state $\varphi(p)^{-1}\varphi(p \cdot p)$ such that the centralizer is trivial. This forces $pMp$ to be trivial and $p$ is minimal in $M$.

	(2) We may assume that $M$ is of type III. Let $\Lambda \leq \R^*_+$ be the subgroup generated by all eigenvalues of $\Delta_\varphi$. By regarding $\Lambda$ as a discrete group, consider the dual compact group  $G:=\widehat{\Lambda}$, so that there is a natural continuous homomorphism $\widehat{\beta}\colon \R \to G$ given by $\widehat{\beta}(t)(\lambda):= \lambda^{{\rm i} t}$ for $t\in \R$, $\lambda \in \Lambda$. 
By \cite[Proposition 1.1]{Co74}, there is a $G$-action $\sigma^{\varphi,\Lambda}
$ on $M$ such that $\sigma^{\varphi,\Lambda}_{\widehat{\beta}(t)} = \sigma^\varphi_t$ for all $t\in \R$. Define 
	$$D_{\varphi,\Lambda}(M):= M \rtimes_{\sigma^{\varphi,\Lambda}} G$$
with the dual weight $\widehat{\varphi}$. Let $h$ be the infinitesimal generator of the unitary representation $\R \ni t \mapsto \lambda^G_{\widehat{\beta}(t)}\in LG$ and put $\Tr:=\widehat{\varphi}(h^{-1} \, \cdot \, )$. Since $\sigma^{\Tr}_t = \id$ for all $t\in \R$, it is indeed a semifinite trace. Observe that the dual action $\widehat{\sigma}$ of $\sigma^{\varphi,\Lambda}$ scales $\Tr$ and therefore it is properly outer, which means that the commutant of $D_{\varphi,\Lambda}(M)$ in $D_{\varphi,\Lambda}(M)\rtimes_{\widehat{\sigma}} \Lambda$ is contained in $D_{\varphi,\Lambda}(M)$. Finally by the Takesaki duality, we have
	$$ D_{\varphi,\Lambda}(M)\rtimes_{\widehat{\sigma}} \Lambda  \simeq M \ovt \B(\ell^2(\Lambda)) \simeq M.$$
Thus $M$ admits a desired semifinite von Neumann subalgebra $D_{\varphi,\Lambda}(M)$ with expectation. The last statement follows by \cite[Theorem 3.3]{Po81}.

Here is another short proof of item (2), by assuming that $M$ is a type III factor with separable predual. By (the proof of) \cite[Theorem 3.5]{HI15}, the centralizer of $\rB(M, \varphi)$ is trivial, while $\varphi$ is almost periodic on $\rB(M, \varphi)$. By item (1), this implies $\rB(M, \varphi) = \C$.
\end{proof}

\subsection*{$q$-deformed Araki--Woods algebras}

	Let $q\in [-1,1]$ and let $H$ be a Hilbert space with a (second linear) inner product $\langle \cdot,\cdot \rangle_H$. We first recall the $q$-deformed Fock space \cite{BS90}.  We define inner products on all algebraic tensor products $H^{\otimes_{\rm alg} n}$ for $n\in \N$ by
	$$\langle \xi_1\otimes\cdots \otimes \xi_n,\eta_1\otimes\cdots \otimes \eta_n\rangle_q = \sum_{\pi\in \mathfrak{S}_n} q^{i(\pi)}\langle \xi_1,\eta_{\pi(1)}\rangle_H \cdots \langle \xi_n,\eta_{\pi(n)}\rangle_H,$$
where $\mathfrak{S}_n$ is the symmetric group and $i(\pi)$ is the number of inversions. They are non-degenerate for all $n\in \N$ and we write as $H^{\otimes n}$ the completion by this inner product. We have the \textit{$q$-deformed Fock space} $\mathcal F_q(H)$ as an infinite direct sum of Hilbert spaces (with the natural inner product):
	$$ \C \Omega \oplus \bigoplus_{n\geq 1} H^{\otimes n} . $$
For each $\xi \in H$, the left and right creation operators $\ell_q(\xi)$ and $r_q(\xi)$ on $\mathcal F_q(H)$ are defined in a natural way. It is known that for all $\xi \in H$,
	\begin{align*}
\|\ell_q(\xi)\|_\infty=\|r_q(\xi)\|_\infty = \frac{1}{\sqrt{1-q}} \|\xi\|_H & \quad (0\leq q<1) \\
\|\ell_q(\xi)\|_\infty=\|r_q(\xi)\|_\infty = \|\xi\|_H \; \quad  \quad \quad & (-1\leq q\leq 0)
\end{align*}
and hence $\ell_q(\xi),r_q(\xi)\in \B(\mathcal F_q(H))$. They are unbounded if $q=1$.

	Assume $q\in (-1,1)$. We next recall the $q$-deformed Araki--Woods algebra \cite{Hi02}. Let $H_\R$ be a real Hilbert space and $(U_t)_t$ a strongly continuous orthogonal representation of $\R$ on $H_\R$. Consider $H:=H_\R \otimes_\R \C$ with the natural $\C$-valued inner product and with the natural involution $I$. We extend $U_t$ by $U_t \otimes \id_\C$ and have a unitary representation of $\R$ on $H$. Let $A$ be the infinitesimal generator of $U_t$, that is, $U_t = A^{ {\rm i} t}$ for all $t\in \R$. We consider the following new embedding 
	$$j \colon H \to H; \ j(\xi)=\frac{\sqrt{2}}{\sqrt{1+A^{-1}}}\, \xi. $$
Put $K_\R := j(H_\R)$ and consider $\mathcal F_q(H)$. For $\xi\in K_\R + {\rm i} K_\R$, we define 
	$$W_q(\xi):=\ell_q(\xi) + \ell_q(T\xi)^*, \quad W_q^r(I\xi):=r_q(I\xi) + r_q(IT{\xi})^* ,$$
 where $T:=IA^{-1/2}$. The \textit{$q$-deformed Araki--Woods algebra} is defined by
	$$ \Gamma_q(H_\R,U_t)\dpr:= \mathrm{W}^* \{ W_q(\xi) \mid \xi \in K_\R \} \subset \B(\mathcal F_q(H)). $$
When $q=0$, this coincides with Shlyakhtenko's free Araki--Woods factors \cite{Sh97}. For simplicity we write as $M_q:=\Gamma_q(H_\R,U_t)\dpr$ and $\varphi_q:= \langle \Omega , \, \cdot \, \Omega \rangle_q$. It is known that $\Omega$ is a cyclic and separating vector for $M_q$ and the commutant is given by $ M_q'= \mathrm{W}^* \{ W_q^r(I\xi) \mid \xi \in K_\R \}$. For any $\xi_1, \ldots, \xi_n\in K_\R + {\rm i} K_\R$, there is a unique element $W_q(\xi_1\otimes \cdots \otimes \xi_n)\in M_q$ such that 
	$$ W_q(\xi_1\otimes \cdots \otimes \xi_n)\Omega =\xi_1\otimes \cdots \otimes \xi_n .$$
Here $W_q^r(I\xi_1\otimes \cdots \otimes I\xi_n)$ is defined similarly.  The modular theory for $(M_q,\varphi_q)$ is given as follows: for all $t\in \R$ and $\xi_1, \ldots, \xi_n\in K_\R$,
\begin{align*}
	& \sigma_t^{\varphi_q}(W_q(\xi_1\otimes \cdots \otimes \xi_n)) = W_q((U_{-t}\xi_1)\otimes \cdots \otimes (U_{-t}\xi_n));\\
	& \Delta_{\varphi_q}(\xi_1\otimes \cdots \otimes \xi_n) = (A^{-1}\xi_1)\otimes \cdots \otimes (A^{-1}\xi_n);\\
	& J_{\varphi_q}(\xi_1\otimes \cdots \otimes \xi_n) = (I\xi_n)\otimes \cdots \otimes (I\xi_1) = (A^{-\frac{1}{2}}\xi_n)\otimes \cdots \otimes (A^{-\frac{1}{2}}\xi_1).
\end{align*}
Recall that any $(H_\R,U_t)$ has a unique decomposition $H_\R = H_\R^{\rm ap} \oplus H_\R^{\rm wm}$, where $(U_t)_t$ acts on $H_\R^{\rm ap}$ as an almost periodic action and on $H_\R^{\rm wm}$ as a weakly mixing action. In this case, the state $\varphi_q$ is almost periodic if and only if  $H_\R^{\rm wm} = 0$.

\section{$\omega$-equicontinuity and ultraproducts}\label{section:equicontinuity}

In this section, we introduce a von Neumann algebra which is contained in the ultraproduct of a given $q$-Araki--Woods algebras. This will be very useful in our work.

Let $\omega$ be any free ultrafilter on $\N$. Let $\cH$ be a real or a complex Hilbert space and $U : \R \curvearrowright \cH$ any strongly continuous orthogonal or unitary representation.

\begin{df}
We say that a bounded sequence $(\xi_n)_n \in \ell^\infty(\N, \cH)$ is $(U, \omega)$-{\em equicontinuous} if for any $\varepsilon > 0$, there exists $ \delta > 0$ such that 
	$$\left \{ n \in \N \mid \sup_{|t| \leq \delta} \|U_t \xi_n - \xi_n\| < \varepsilon \right \} \in \omega.$$
We denote by $\mathfrak E(\cH, U, \omega)$ the subspace of $\ell^\infty(\N, \cH)$ consisting in all $(U, \omega)$-equicontinuous bounded sequences.
\end{df}

\begin{example}
Here are basic examples of $(U, \omega)$-equicontinuous bounded sequences.
\begin{enumerate}
\item Let $(\xi_n)_n \in \ell^\infty(\N, \cH)$ be any $U$-{\em almost invariant} bounded sequence. By assumption, for every $\kappa > 0$, we have $\lim_n (\sup_{|t|  \leq \kappa} \|U_t \xi_n - \xi_n\|) = 0$. In particular, for any $\varepsilon > 0$ and any $\delta > 0$, we have $\{ n \in \N \mid \sup_{|t| \leq \delta} \|U_t \xi_n - \xi_n\| < \varepsilon \} \in \omega$. Thus, $(\xi_n)_n$ is $(U, \omega)$-equicontinuous.

\item Let $\xi \in \cH$ be any vector and $(t_n)_{n}$ any sequence of reals. For every $n \in \N$, put $\xi_n = U_{t_n} \xi$. For every $\varepsilon > 0$, there exists $\delta > 0$ such that $\sup_{|t| \leq \delta} \|U_t \xi - \xi\| < \varepsilon$. Then for every $n \in \N$, we have 
$$\sup_{|t| \leq \delta} \|U_t \xi_n - \xi_n\| = \sup_{|t| \leq \delta} \|U_{t_n} (U_t \xi - \xi)\| = \sup_{|t| \leq \delta} \|U_t \xi - \xi\| < \varepsilon.$$
Thus, $(U_{t_n}\xi)_n$ is $(U, \omega)$-equicontinuous.
\end{enumerate}
\end{example}

Put $\mathfrak I_\omega(\cH) = \{(\xi_n)_n \in \ell^\infty(\N, \cH) \mid \lim_{n \to \omega} \|\xi_n\| = 0\}$. Observe that 
	$$\mathfrak I_\omega(\cH) \subset \mathfrak E(\cH, U, \omega) \subset \ell^\infty(\N, \cH).$$ 
Denote by $\cH_\omega = \ell^\infty(\N, \cH) / \mathfrak I_\omega(\cH)$ the ultraproduct Hilbert space. For every $(\xi_n)_n\in \ell^\infty(\N, \cH)$ and $t \in \R$, denote by $(\xi_n)_\omega$ the image in $\cH_\omega$, and by $(U_\omega)_t $ the orthogonal or unitary transformation on $\cH_\omega$ defined by $(U_\omega)_t (\xi_n)_\omega = (U_t \xi_n)_\omega$. Observe however that the map $\R\ni t \mapsto (U_\omega)_t $ need not be strongly continuous. 
Denote by $\cH_{U, \omega} \subset \cH_\omega$ the Hilbert subspace defined by $\cH_{U, \omega} = \mathfrak E(\cH, U, \omega)/\mathfrak I_\omega(\cH)$. Then the restriction of $(U_\omega)_t$ on $\cH_{U,\omega}$ is strongly continuous. Observe that $\R \curvearrowright \cH$ is a subrepresentation of $\R \curvearrowright \cH_{U, \omega}$.

Now let $(H_\R, U_t)$ be any strongly continuous orthogonal representation. We keep the notation from Section \ref{Preliminaries}, such as $j,K_\R,A,\varphi_q$, and $M_q:=\Gamma_q(H_\R,U_t)\dpr$. Observe that there is a canonical identification $(H_\R)_{U,\omega}\otimes_\R \C = H_{U,\omega}$. 
Let $A_{\omega}$ be the infinitesimal generator of $U_\omega$ on $H_{U,\omega}$. For any $a>0$, define $\mathcal{D}(a)\subset H_{U,\omega}$ as the subspace of all vectors $\xi$ such that $\xi$ has a representative $(\xi_n)_n$ with $\xi_n\in E_{\log(A)}([-a,a])H$, where $E_{\log(A)}([-a,a])$ is the spectral projection of $\log(A)$ for $[-a,a]$. By \cite[Theorem 4.1]{AOS13} (see also \cite{KZ85,AH12}), it holds that
\begin{itemize}
	\item $\mathcal{D}_0:=\bigcup_{a\in\N} \mathcal{D}(a)$ is a core of $\log (A_\omega)$;

	\item $\log (A_\omega) (\xi_n)_\omega = (\log(A)\xi_n)_\omega$ for all $(\xi_n)_\omega\in \mathcal{D}_0$.

\end{itemize}
These imply that $f(\log (A_\omega)) (\xi_n)_\omega = (f(\log(A))\xi_n)_\omega$ for all $(\xi_n)_\omega\in \mathcal{D}_0$ and all polynomials $f$. Since $\log (A_\omega)$ is bounded on $\mathcal{D}(a)$ for each $a$, this equation holds for all $f\in C(\R)$. In particular, the cases $f(\log(A_\omega))=A_\omega$ and $f(\log(A_\omega))=[2/(1+A_\omega^{-1})]^{-1/2}$ hold on $\mathcal D_0$. The second case means that
	$$ j_\omega((\xi_n)_\omega) = \frac{\sqrt{2}}{\sqrt{1+A_\omega^{-1}}} (\xi_n)_\omega =  \left(\frac{\sqrt{2}}{\sqrt{1+A^{-1}}}\xi_n\right)_\omega = (j(\xi_n))_\omega,\quad \text{for all }(\xi_n)_\omega\in \mathcal{D}_0,$$
where $j_\omega$ is the $j$-map for $A_\omega$. Since $j_\omega$ and $j$ are bounded, this holds for all $(\xi_n)_\omega\in H_{U,\omega}$.

\begin{prop}
	Let $(\xi_n)_n \in \ell^\infty(\N, H)$ be any bounded sequence such that $\xi_n\in K_\R$ for all $n$. The following conditions are equivalent:
\begin{enumerate}
\item [$(\rm 1)$]  $(\xi_n)_n$ is $(U, \omega)$-equicontinuous.

\item [$(\rm 2)$] $(W_q(\xi_n))_n$ is $(\sigma^{\varphi_q}, \omega)$-equicontinuous in the sense of \cite[Subsection 1.4]{MT13}.

\item [$(\rm 3)$] $(W_q(\xi_n))_n \in \mathfrak M^\omega(M_q)$.
\end{enumerate}
If $(\eta_n)_n \in \ell^\infty(\N, H_\R)$ is any $(U, \omega)$-equicontinuous sequence, then $(j(\eta_n))_n$ is $(U, \omega)$-equicontinuous, hence it satisfies all above conditions.
\end{prop}

\begin{proof}
Observe that $W_q(\eta) = W_q(\eta)^*$ is selfadjoint for every $\eta \in K_\R$. The equivalence $(\rm 1) \Leftrightarrow (\rm 2)$ follows from the facts that for every $t \in \R$ and every $n \in \N$, we have 
	$$\|U_t \xi_n - \xi_n\|_{H} = \|W_q(U_t \xi_n) - W_q(\xi_n)\|_{\varphi_q} = \|\sigma_{-t}^{\varphi_q}(W_q(\xi_n)) - W_q(\xi_n)\|_{\varphi_q}$$ and that on uniformly bounded sets, the norm $\|\cdot\|_\varphi$ induces the strong operator topology. 

The equivalence $(\rm 2) \Leftrightarrow (\rm 3)$ follows from \cite[Theorem 1.5]{MT13}. The last statement is trivial.
\end{proof}

The next result shows that $\Gamma_q((H_\R)_{U, \omega}, U_\omega)\dpr$ canonically embeds into the ultraproduct of $ \Gamma_q(H_\R, U)\dpr$.

\begin{thm}\label{thm-ultraproduct}
	Let $q\in (-1,1)$ and let $(H_\R, U_t)$ be any strongly continuous orthogonal representation. We denote by $(M_q,\varphi_q)$ the associated $q$-Araki--Woods algebra with vacuum state. Then the map
	$$\iota : \Gamma_q((H_\R)_{U, \omega}, U_\omega)\dpr \to M_q^\omega : W_q(j_\omega((\xi_n)_\omega)) \mapsto (W_q(j(\xi_n)))_\omega$$
extends to a unital $\ast$-embedding which preserves the vacuum states.

After identifying $\Gamma_q((H_\R)_{U, \omega}, U_\omega)\dpr$ with its image $\iota(\Gamma_q((H_\R)_{U, \omega}, U_\omega)\dpr)$ in $M_q^\omega$, we have the following $\sigma^{\varphi_q^\omega}$-invariant inclusions
	$$M_q \subset \Gamma_q((H_\R)_{U, \omega}, U_\omega)\dpr \subset M_q^\omega.$$
\end{thm}

\begin{proof}
The proof is straightforward by the previous proposition.
\end{proof}

To use this theorem in our work, we will need the following lemma.

\begin{lem}\label{lem-ultraproduct}
	The following statements hold true.
\begin{enumerate}
	\item [$(\rm 1)$] We have $\sigma(A) \setminus \{0\} \subset \sigma_p(A_\omega)$, where $\sigma$ (resp.\ $\sigma_p$) denotes the spectrum (resp.\ point spectrum).  

	\item	 [$(\rm 2)$] If $(H_\R,U_t)$ is weakly mixing, then there is a separable closed subspace $L_\R\subset (H_\R)_{U,\omega} \ominus H_\R$ such that the restriction of $U_\omega$ on $L_\R$ is almost periodic and that $\Gamma_q(L_\R , U_\omega)\dpr$ is a non-amenable factor.
\end{enumerate}
\end{lem}

\begin{proof}
	$(\rm 1)$ This is a well known fact but we include a short proof. Let $\lambda \in \sigma(A) \setminus \{0\}$. Then $\log(\lambda) \in \sigma(\log(A))$. Choose a sequence of unit vectors $\xi_n \in \mathcal D(\log(A))$ such that $\lim_n \|(\log(A) - \log(\lambda)) \xi_n \|= 0$. Let $\varepsilon > 0$ and choose $\delta > 0$ such that $\delta |\log(\lambda)| \leq \varepsilon/2$. For every $t \in [-\delta, \delta]$ and every $x \in \R$, since $|\exp({\rm i}t x) - 1| \leq \delta |x|$, it follows that
\begin{align*}
	\sup_{t \in [-\delta, \delta]} \|U_t \xi_n - \xi_n\| 
	&\leq \sup_{t \in [-\delta, \delta]} \|\lambda^{- {\rm i} t}U_t \xi_n - \xi_n \| +  \sup_{t \in [-\delta, \delta]} |\lambda^{{\rm i}t} - 1| \\
	&\leq \sup_{t \in [-\delta, \delta]}  \|(\exp({\rm i} t (\log(A) - \log(\lambda))) -1) \xi_n\| + \delta |\log(\lambda)|\\
	&\leq \delta \|(\log(A) - \log(\lambda)) \xi_n\| + \varepsilon/2 \to \varepsilon/2 \quad (\text{as }n\to \infty).
\end{align*}
This implies that $(\xi_n)_n$ is $(U, \omega)$-equicontinuous and we can define $\xi = (\xi_n)_\omega \in H_{U, \omega}$. We have $(U_\omega)_t \xi = \lambda^{{\rm i}t} \xi$ for every $t \in \R$ and this shows that $\lambda \in \sigma_p(A_\omega)$.

	$(\rm 2)$ Let $\lambda \in \sigma_p(A_\omega)$ and  $\xi = (\xi_n)_\omega \in H_{U, \omega}$ be any unit vector such that $\lambda^{{\rm i}t} \xi = (U_\omega)_t \xi$ for every $t \in \R$. Since $U$ is weakly mixing, it follows that $\xi_n \to 0$ weakly as $n \to \omega$ and so $\xi \in H_{U, \omega} \ominus H$. 

Since $U$ is weakly mixing, there is an infinite subset $\{\lambda_n\}_{n\in \N} \subset \sigma(A) \cap (1,\infty)$ such that $\sup_{n\in \N}\lambda_n <\infty$. For any $\lambda_n$, fix a unit eigenvector $\xi_n \in H_{U, \omega} \ominus H$. One can then find a unique real 2-dimensional subspace $L_\R^n \subset (H_\R)_{U, \omega} \ominus H_\R$ such that $\xi_n \in L_\R^n \otimes_\R \C$ and that $U_\omega$ acts on $L_\R^n \otimes_\R \C$ with eigenvalue $\lambda_n$ (and $\lambda_n^{-1}$). Define a closed subspace 
	$$ L_\R:=\bigoplus_{n\in \N} L_\R^n \subset (H_\R)_{U, \omega} \ominus H_\R,$$
so that the restriction of $U_\omega$ on $L_\R$ is almost periodic. By \cite[Theorems 2.2 and 3.2]{Hi02}, since $\sup_{n\in \N}\lambda_n <\infty$, $\Gamma_q(L_\R , U_\omega)\dpr$ is a non-amenable factor.
\end{proof}

\section{Two key lemmas}

In this section, we prove two key lemmas. 
Throughout this section, we fix $q\in (-1,1)$ and any strongly continuous orthogonal representation $(H_\R,U_t)$. We keep the same notation as in previous sections, such as $M_q = \Gamma_q(H_\R,U_t)\dpr$, $\varphi_q=\langle \Omega, \, \cdot \, \Omega \rangle$ and $j\colon H_\R \to H$. 
For any real subspace $D_\R \subset H_\R$ that is globally preserved by $(U_t)_t$, we let $M_q(D_\R) := \Gamma_q(D_\R,U_t)\dpr$ and regard it as a subalgebra of $M_q$ with expectation. 
For each $N\in \N$, we denote by $P_N\in \B(\mathcal F_q(H))$ the orthogonal projection onto $H^{\otimes N}$ and put $P_{\leq N} := \sum_{k=0}^{N} P_k$. We fix a free ultrafilter $\omega$ on $\N$.

\subsection*{A lemma for commutants}

The first key lemma is for commutants of subalgebras arising from real subspaces. The proof here is strongly inspired by  \cite[Theorem 3]{SW16} (see also \cite{Ri04}).

\begin{lem}\label{commutant lemma}
	Let $D_\R \subset H_\R$ be a real subspace which is globally preserved by $(U_t)_{t}$ and let $p\in M_q(D_\R)$ be any nonzero projection. Assume that there is a unitary $u=(u_n)_\omega \in pM_q(D_\R)^\omega p$ and $\delta\in [0,1)$ such that
	$$ \lim_{n\to \omega} \| P_{\leq N} u_n \Omega \|_q \leq \delta ,\quad \text{for all }N\in \N.$$
Then any $x\in pM_q p$ which commutes with $u$ satisfies 
	$$ \|x - E_{D_\R}(x) \|_{\varphi_q} \leq \delta \|x - E_{D_\R}(x) \|_\infty ,$$
where $E_{D_\R}\colon M_q \to M_q(D_\R)$ is the $\varphi_q$-preserving conditional expectation.
\end{lem}
\begin{proof}
	Put $x^\circ := x-E_{D_\R}(x)$ and fix $\xi \in \mathcal F_q(D)^\perp \subset \mathcal F_q(H)$. We will compute $\langle \xi,x^\circ \Omega \rangle_q$. We assume that $\xi$ is contained in a linear span of elements of the form that 
	$$\eta_1\otimes \cdots \otimes \eta_n, \quad \text{for some }n\in \N, \ \eta_k\in Ij(H_\R), \ 1\leq k \leq n. $$
We can find a unique element $W_q^r(\xi)\in M_q'$ such that $W_q^r(\xi)\Omega = \xi$. Since $x^\circ $ also commutes with $u$, one has
	$$\langle \xi,x^\circ \Omega \rangle_q = \langle u\xi,x^\circ  u \Omega \rangle_q= \langle W_q^r(\xi)u\Omega ,x^\circ  u \Omega \rangle_q ,$$
where the inner product is interpreted in $\rL^2(M_q^\omega)$ (e.g.\ $\langle u\xi,x^\circ  u \Omega \rangle_q=\lim_{n\to \omega}\langle u_n\xi,x^\circ  u_n \Omega \rangle_q$) and hence
	$$\langle \xi,x^\circ \Omega \rangle_q= \lim_{N\to \infty} \left(\langle W_q^r(\xi)u\Omega ,x^\circ   P_{\leq N} u \Omega \rangle_q + \langle W_q^r(\xi) u\Omega ,x^\circ  P_{\leq N}^\perp u \Omega \rangle_q\right),$$
where $P_{\leq N}^\perp = 1- P_{\leq N}$. 
We denote by $\Phi_\xi(a) := \lim_{N\to \infty} \langle W_q^r(\xi) u\Omega , a  P_{\leq N}^\perp u \Omega \rangle_q$, which corresponds to the second term in the above limit. We claim that  $\Phi_\xi(a)=0$ for all $a\in M_q$.

	To see the claim, since $u$ is in the ultraproduct, the map 
	$$M_q\ni a \mapsto a W_q^r(\xi)u\Omega\in M^\omega_q M_q' \Omega \subset \rL^2(M^\omega_q)$$
is continuous (from the strong topology to the norm topology) and hence we may assume that $a$ is of the form $a=W_q(\eta)$ for some $\eta=\eta_1\otimes \cdots \otimes \eta_m$, $\eta_k\in j(H_\R)$, $k=1,\ldots,m$. As in the proof of \cite[Theorem 3]{SW16}, we have only to check the following condition:
	$$ \lim_{N\to \infty} \left\|\left(\sum_{k=0}^\infty q^k P_{k}\right) X P_{\leq N}^\perp u \Omega\right\|_q =0 ,$$
where $X$ is any element which is contained in the $\ast$-algebra generated by $\ell_q(e),r_q(f)$ for some $e,f\in H$. To see this, observe that $X$ changes lengths in the Fock space at most finitely many times and therefore we can write $X P_{\leq N}^\perp = P_{\leq (N - c(X))}^\perp X P_{\leq N}^\perp$ for some $c(X)>0$. Then we can use $\lim_{N\to \infty}\|\sum_{k=0}^\infty q^k P_{k} P_{\leq (N - c(X))}^\perp\|_\infty = 0$ to deduce the above convergence. 
We conclude that $\Phi_\xi(a)=0$ for all $a\in M_q$.

	Now we have the following equality
	$$\langle \xi,x^\circ \Omega \rangle_q = \lim_{N\to \infty}\langle W_q^r(\xi) u\Omega ,x^\circ   P_{\leq N} u\Omega \rangle_q = \lim_{N\to \infty}\langle  (x^\circ )^*u \xi , P_{\leq N} u\Omega \rangle_q.$$
Combined with the assumption on $u$, we obtain 
	$$|\langle \xi,x^\circ \Omega \rangle_q| \leq \delta \|x^\circ \|_\infty \|\xi\|_q.$$
By the choice of $\xi$, this inequality actually holds for all $\xi\in\mathcal F_q(D)^\perp$. Since $x^\circ \Omega \in \mathcal F_q(D)^\perp$  we obtain the desired condition $ \| x^\circ  \Omega \|_q \leq \delta \|x^\circ \|_\infty$.
\end{proof}

The next proposition is a special case of the previous lemma. This slightly generalizes \cite[Proposition 6.(2)]{SW16} and should be compared to  \cite[Proposition 3.1]{Ue10}.

\begin{prop}
	Let $D_\R \subset H_\R$ be a real subspace which is globally preserved by $(U_t)_{t}$. Assume that $D_\R$ is finite dimensional. 
Then for any projection $p\in M_q(D_\R)$ and any diffuse von Neumann subalgebra $A \subset pM_q(D_\R)p$ with expectation, we have $A' \cap pM_q p \subset pM_q(D_\R)p$.
\end{prop}
\begin{proof}
	Let $\psi$ be any faithful normal state on $A$ such that $A_\psi$ is diffuse. Take any sequence $u_n \in \mathcal{U}(A_\psi)$ such that $u_n \to0$ weakly. Observe that $\lim_{n\to \infty}\|P_{\leq N}u_n \Omega\|_q \to 0$ for any $N\in \N$, since $D_\R$ is finite dimensional. Since $(u_n)_n$ determines a unitary element in the ultraproduct $A_\psi^\omega \subset p M_q(D)^\omega p$, we can use Lemma \ref{commutant lemma} for the case $\delta=0$ and get the conclusion.
\end{proof}

\subsection*{A lemma for unitaries in ultraproducts}

Our second key lemma is to find a nice sequence of unitaries which gives an element in ultraproducts. We prepare a terminology and a lemma.

	Let $\C \langle X_1,Y_1, \ldots, X_n,Y_n\rangle$ denote the set of all non-commutative polynomials of $2n$-variables. For any fixed $\lambda_1,\ldots,\lambda_n\in [1,\infty)$, consider the unique action $\sigma$ of $\R$ on $\C \langle X_1,Y_1, \ldots, X_n,Y_n\rangle$ given by, for any $p\in \C \langle X_1,Y_1, \ldots, X_n,Y_n\rangle$ and $t\in \R$,
	$$ \sigma_t(p(X_1,Y_1, \ldots, X_n,Y_n)) = p(\lambda_1^{\ri t}X_1,\lambda_1^{-\ri t}Y_1, \ldots, \lambda_n^{\ri t}X_n,\lambda_n^{-\ri t}Y_n).$$
We will say that $p\in \C \langle X_1,Y_1, \ldots, X_n,Y_n\rangle$ is \textit{stable under $(\lambda_1,\ldots,\lambda_n)$-perturbation} if it is fixed by this action. 
This amounts to saying that $p$ is contained in the linear spans of all monomials $X_{i_1}^{\varepsilon_1}Y_{j_1}^{\delta_1}\cdots X_{i_m}^{\varepsilon_m}Y_{j_m}^{\delta_m}$ such that $\lambda_{i_1}^{\varepsilon_1}\lambda_{j_1}^{-\delta_1}\cdots \lambda_{i_m}^{\varepsilon_m}\lambda_{j_m}^{-\delta_m} = 1$, where $m\in \N$, $i_1,\ldots,i_m,j_1,\ldots,j_m \in \{1,\ldots,n\}$, and $\varepsilon_1,\ldots, \varepsilon_m, \delta_1,\ldots,\delta_m \in \N \cup \{0\}$. 

\begin{lem}\label{key lemma1}
	Assume that $\dim H_\R = 2 n< \infty$ with the decomposition $H_\R = \bigoplus_{i=1}^n H_\R^i$ such that $\dim{H_\R^i} =2$ and $(U_t)_t$ acts on $H_\R^i\otimes_\R \C$ with eigenvalue $\lambda_i\geq 1$ for all $i$. For each $i$, 
let $\xi_i\in H_\R^i \otimes_\R \C $ be a unit eigenvector for $\lambda_i$. Let $r \in (M_q)_{\varphi_q}$ be a nonzero projection such that $rM_qr$ is diffuse. 

Then for any $\varepsilon>0$ and $N\in \N$, there is a non-commutative polynomial $p\in \C \langle X_1,Y_1, \ldots, X_n,Y_n\rangle$, which is stable under $(\lambda_1,\ldots,\lambda_n)$-perturbation, satisfying the following conditions: with the notation $p_\xi:=p(W_q(\xi_1),W_q(\xi_1)^*, \ldots, W_q(\xi_n),W_q(\xi_n)^*) \in M_q$,
\begin{itemize}
	\item[$\rm (i)$] $\| p_\xi \|_\infty\leq 1$;

	\item[$\rm (ii)$] $\|P_{\leq N} p_\xi \Omega \|_q < \varepsilon$;

	\item[$\rm (iii)$] $\|(p_\xi^*p_\xi-r)\Omega\|_{q}+\|(p_\xi p_\xi^*-r)\Omega\|_{q}< \varepsilon$.
\end{itemize}
\end{lem}
\begin{proof}
	Observe that $r(M_q)_{\varphi_q}r$ is diffuse by Lemma \ref{almost periodic lemma}(1). So there is a sequence $u_n\in \mathcal{U}(r(M_q)_{\varphi_q}r)$ which converges to $0$ weakly. Since $\dim H_\R < \infty$,  $\lim_{n\to \infty}\|P_{\leq N} u_n \Omega\|_q=0$ for all $N\in \N$. 
Thus for any fixed $\varepsilon>0$ and $N\in \N$, we can choose $u\in \mathcal{U}(r(M_q)_{\varphi_q}r)$ such that $ \|P_{\leq N} u \Omega \|_q < \varepsilon$. We will approximate $u$ by algebraic elements.

Recall that for each $k$, $T{\xi_k}:=IA^{-1/2}\xi_k$ is an eigenvector for $\lambda_k^{-1}$ and $W_q(\xi_k)^* = W_q(T{\xi_k})$. Define $\mathcal A$ as a set of all 
	$$ p(W_q(\xi_1),W_q(\xi_1)^*, \ldots, W_q(\xi_n),W_q(\xi_n)^*) \in M_q,$$
where $p\in \C \langle X_1,Y_1, \ldots, X_n,Y_n\rangle$ is any non-commutative polynomial which is stable under $(\lambda_1,\ldots,\lambda_n)$-perturbation. It is easy to see that $\mathcal A$ is a unital $\ast$-subalgebra in $(M_q)_{\varphi_q}$ and that $\mathcal A \Omega$ is norm dense in $\rL^2((M_q)_{\varphi_q},\varphi_q)$. Since $\varphi_q$ is a trace on $(M_q)_{\varphi_q}$, we obtain that $\mathcal A$ is $\sigma$-weakly dense in $(M_q)_{\varphi_q}$. 
Thus the above unitary $u$ can be approximated by elements in $\mathcal A$ in the $\ast$-strong topology. We can find a non-commutative polynomial $p$, which is stable under $(\lambda_1,\ldots,\lambda_n)$-perturbation, such that $u$ and $p_\xi:= p(W_q(\xi_1),W_q(\xi_1)^*,\ldots,W_q(\xi_n),W_q(\xi_n)^*)$ are close enough in the $\ast$-strong topology and that $\|p_\xi\|_\infty\leq \|u\|_\infty=1$. This is the conclusion.
\end{proof}

Now we prove the second key lemma.

\begin{lem}\label{key lemma2}
	Let $H_\R = H_\R^{\rm ap}\oplus H_\R^{\rm wm}$ be the unique decomposition into the almost periodic and the weakly mixing part. Assume that $\dim_{\R}H_\R^{\rm ap} = \infty $ and $\Gamma_q(H_\R^{\rm ap},U_t)\dpr$ is non-amenable. Then there is a sequence $u_n \in \mathcal U( (M_q)_{\varphi_q})$ such that
$$\lim_{n\to \omega} \|P_{\leq N} u_n \Omega\|_q=0, \quad \text{for all }N\in \N.$$
\end{lem}

\begin{proof}
	Observe that $H_\R^{\rm ap}$ contains a countably infinite dimensional real subspace which is globally invariant under $(U_t)_t$ such that the associated $q$-Araki--Woods algebra is non-amenable. Hence up to exchanging $H_{\R}$ with this subspace, we may assume $H_\R$ is almost periodic and countably infinite dimensional. It has a decomposition $ H_\R = \bigoplus_{n\in \N} H_\R^n $ such that each $H_\R^n$ is 2-dimensional with eigenvalue $\lambda_n\geq 1$. Define a family of $q$-Araki--Woods algebras by 
	$$ N_n:= \Gamma_q(\oplus_{i=1}^n H_\R^i, U_t)\dpr, \quad \text{for all }n\in \N . $$
Observe that we have canonical inclusions $N_n \subset N_{n+1}\subset M_q$ with (vacuum state preserving) expectation for all $n\in \N$ and that $\bigcup_{n\in \N} N_n $ is $\sigma$-weakly dense in $M_q$. 
For each $n\in \N$, we take the unique central projection $z_n \in \mathcal {Z}(N_n) \subset M_q$ such that $N_n z_n$ has no amenable direct summand and $N_n z_n^{\perp}$ is amenable, where $z_n^\perp:=1-z_n$. 

\begin{claim}
	The sequence $(z_n)_n$ is increasing and it converges to $1$ strongly. 
\end{claim}
\begin{proof}[Proof of Claim]
	Fix $n\in \N$ and observe that $N_nz_{n+1}^\perp$ is amenable, since $N_{n+1}z_{n+1}^\perp$ is amenable. Since $z_{n+1}^\perp$ is contained in $N_n'$, if we denote by $w\in \mathcal Z(N_{n})$ the central support projection of $z_{n+1}^\perp$ in $N_n'$, then $N_n w$ is also amenable. We conclude that $w \leq z_n^\perp$ by the choice of $z_n$ and hence $z_{n+1}^\perp \leq z_n^\perp$. Thus $(z_n)_n$ is increasing. 
We denote by $z$ the strong limit of the increasing sequence $(z_n)_n$. Observe that $z$ is a central projection in $M_q$. Since $M_q$ is a factor by \cite[Theorem 3.2]{Hi02}, $z$ is $0$ or $1$. Observe that $M_q z^\perp = \left(\bigcup_n N_nz^\perp\right)''$ is amenable. Since we assumed $M_q$ is non-amenable, we conclude that $z^\perp=0$.
\end{proof}

Now we fix any $\varepsilon>0$ and $N\in\N$. By the claim, there is $k\in \N$ such that $\|(1-z_k)\Omega \|_q < \varepsilon/3$. We fix such $k$ and observe that $N_k z_k$ is diffuse since it has no amenable direct summand. 
We apply Lemma \ref{key lemma1} to $N_k$ and $z_k$ (and $\varepsilon/3$ and $N$) and find $p\in \C \langle X_1,Y_1, \ldots, X_k,Y_k\rangle$, which is stable under $(\lambda_1,\ldots,\lambda_k)$-perturbation, satisfying conditions (i)-(iii) for some fixed unit eigenvectors $\xi_1,\ldots,\xi_k$ and corresponding $p_\xi$. Since $\|(1-z_k)\Omega\|_q < \varepsilon/3$, we indeed have 
\begin{itemize}
	\item[$\rm (iii)'$] $\|(p_\xi^*p_\xi-1)\Omega\|_{q}+\|(p_\xi p_\xi^*-1)\Omega\|_{q}< \varepsilon$.
\end{itemize}
Thus for any $\varepsilon>0$ and $N\in \N$, we found $p_\eta\in (M_q)_{\varphi_q}$ satisfying (i), (ii), and $\rm (iii)'$. If we put $u_n:=p_\xi$ which is taken for $\varepsilon=1/n$ and $N=n$ for all $n\in \N$, then the sequence $(u_n)_n$ determines an element in $\mathcal{U}(((M_q)_{\varphi_q})^\omega)$ by (i) and $\rm(iii)'$, and satisfies the conclusion of this lemma by (ii). 
\end{proof}

\section{Proofs of Main Theorem and Corollary}

\begin{proof}[Proof of Main Theorem]
	We continue to use the notation $M_q (=M_q(H_\R))=\Gamma_q(H_\R,U_t)\dpr$ with vacuum state $\varphi_q$. Let $H_\R=H_\R^{\rm ap}\oplus H_\R^{\rm wm}$ be the unique decomposition into the almost periodic and the weakly mixing part. Since $H_\R^{\rm wm}\neq 0$ by assumption, $M_q$ is a type III$_1$ factor \cite[Theorem 8.1]{BM16}. The fullness will be proved in the last section. So we have only to show that $\mathrm{B}(M_q,\varphi_q)$ is trivial.

Consider the real Hilbert space given by 
	$$ \mathcal H_\R := H_\R^{\rm ap} \oplus H_\R^{\rm wm} \oplus \mathcal K_\R, \quad \text{where } \mathcal K_\R:=(H_\R^{\rm wm})_{U,\omega}\ominus H_\R^{\rm wm}.$$
Using the inclusion in Theorem \ref{thm-ultraproduct}, we regard $M_q(\mathcal H_\R)$ as a subalgebra in $M_q^\omega$. Fix a real separable subspace $L_\R \subset \mathcal K_\R$ as in item $(\rm 2)$ in Lemma \ref{lem-ultraproduct}, so that
\begin{align*}
	\mathrm{B}(M_q,\varphi_q)= M_q \cap ((M_q^\omega)_{\varphi^\omega_q})' 
	\subset M_q(\mathcal H_\R) \cap (M_q(L_\R)_{\varphi^\omega_q})'.
\end{align*}
Let $\omega'$ be any free ultrafilter on $\N$ (one can take $\omega' = \omega$). Then Lemma \ref{key lemma2} shows there exists a sequence $u_n \in M_q(L_\R)_{\varphi_q^\omega}$ such that $\lim_{n \to \omega'} \|P_{\leq N} u_n \xi_{\varphi^\omega_q}\|_q = 0$ for all $N \in \N$. Hence any element in $\mathrm{B}(M_q,\varphi_q)$ commutes with $u_n$ for all $n\in \N$ and therefore it is contained in $M_q(L_\R)$ by Lemma \ref{commutant lemma} (for the case $\delta=0$). We get 
	$$\mathrm{B}(M_q,\varphi_q)=M_q \cap ((M_q^\omega)_{\varphi_q^\omega})' \subset M_q(H_\R)\cap M_q(L_\R)=\C.$$
This is the desired conclusion.
\end{proof}

\begin{proof}[Proof of Corollary]
	Let $H_\R=H_\R^{\rm ap}\oplus H_\R^{\rm wm}$ be the unique decomposition into the almost periodic and the weakly mixing part. If $H_\R^{\rm wm}=0$, then one can directly apply Lemma \ref{almost periodic lemma}(2) and get the conclusion. If $H_\R^{\rm wm}\neq 0$, we can use Main Theorem together with \cite{Ha85, Po81} to get the conclusion.
\end{proof}

\section{Remark on fullness}

	In \cite{Sn03}, it is proved that $q$-Gaussian algebras $\Gamma_q(H_\R,\id)\dpr$ are full factors, provided that $\dim_{\R}H_\R$ is sufficiently large. In the same article, the author claims that the same techniques can be applied to many $q$-Araki--Woods algebras \cite[Section 5]{Sn03}. However, it is not very straightforward to apply his argument to $q$-Araki--Woods algebras. So in this section, we explain it and deduce the fullness of a certain class of $q$-Araki--Woods algebras, which particularly includes the ones arising from weakly mixing actions.

	We fix the following setting. Let $q\in (-1,1)$ and let $(H_\R,U_t)$ be any strongly continuous orthogonal representation on a real Hilbert space with the infinitesimal generator $A$ on $H=H_\R\otimes_\R \C$. Put $\mathcal F_q:=\mathcal F_q(H)$ and $\mathcal F_q^+:=\mathcal F_q(H)\ominus \C\Omega$, where $\Omega$ is the vacuum vector.

\begin{prop}[{\cite[Proposition 1]{Sn03}}]
	Consider linear maps given by 
\begin{align*}
	& \Phi\colon H \otimes \mathcal F_q(H) \to \mathcal F_q^+(H);\quad \xi \otimes (\eta_1\otimes \cdots\otimes \eta_{n})\mapsto \xi\otimes \eta_1\otimes \cdots \otimes \eta_{n}\\
	& \Psi\colon \mathcal F_q(H)\otimes H \to \mathcal F_q^+(H);\quad (\eta_1\otimes \cdots\otimes \eta_{n})\otimes \xi\mapsto \eta_1\otimes \cdots \otimes \eta_{n}\otimes \xi
\end{align*}
Then there are constants $C_1,C_2>0$ (which depend only on $q$) such that 
	$$ \|\Phi\|=\|\Psi\|\leq C_1 \quad \text{and} \quad \|\Phi^{-1}\| = \|\Psi^{-1}\|\leq C_2. $$
\end{prop}

The goal of this section is to prove the following theorem. We note that the assumption of the theorem holds if $(H_\R,U_t)$ contains a weakly mixing subrepresentation.

\begin{thm}\label{full thm}
	Let $q\in (-1,1)$. Let $(H_\R,U_t)$ be any strongly continuous orthogonal representation on a real Hilbert space with the infinitesimal generator $A$ on $H=H_\R\otimes_\R \C$. Assume that there exist constants $d\in \N$ and $C\geq 1$ such that 
	$$ \dim E_A(\{1\})H + 2 \dim E_A((1,C^2])H \geq  d \quad \text{and} \quad {d^2}  > 2(C_1C_2)^2 (8  C ^2 d +  1),$$
where $E_A(X)$ are spectral projections for $X\subset \R$. Then $\Gamma_q(H_\R,U_t)\dpr$ is a full factor.
\end{thm}

\begin{lem}\label{assumption lemma}
	Assume that there are constants $d \in \N$ and $C> 1$ satisfying
	$$ \dim E_A(\{1\})H + 2 \dim E_A((1,C^2])H \geq  d .$$
Then there is an orthonormal family $\{e_i\}_{i=1}^d \subset E_A([C^{-2},C^2])H$ such that $Ie_i=e_i$ for all $i$. In that case, we have
\begin{itemize}
	\item $Te_i=A^{1/2}e_i$ and $ITIe_i=A^{-1/2}e_i$ for all $i$ (where $T:=IA^{-1/2}$);
	\item $\max_i\{\|A^{1/2}e_i\|, \|A^{-1/2}e_i\|\}\leq C$.
\end{itemize}
\end{lem}
\begin{proof}
	We first choose an orthonormal basis $\{e_k\}_k$ from $E_A(\{1\})H$ which are all $I$-invariant. We will extend this to the desired family. For this, up to exchanging $d - \dim E_A(\{1\})H$ by $d$ and $H\ominus E_A(\{1\})H$ by $H$, we may assume $E_A(\{1\})H=0$. We write as $H^{\rm wm}:=H^{\rm wm}_\R \otimes_\R \C$ and $H^{\rm ap}:=H^{\rm ap}_\R \otimes_\R \C$.

	Suppose first that $E_A((1,C^2])H^{\rm wm} \neq 0$. Then the $U_t$-action on $E_A((1,C^2])H^{\rm wm}$ is weakly mixing, so there exist mutually orthogonal vectors $\xi_1,\ldots,\xi_d\in E_A((1,C^2])H^{\rm wm}$. Put
	$$ e_i:= \frac{1}{\|\xi_i + I \xi_i\|}(\xi_i + I \xi_i)\quad (i=1,\ldots,d).$$
Then since $I\xi_i \in E_A([C^{-2},1))H^{\rm wm}$, it is easy to see that $\{e_i\}_{i=1}^d$ is an orthonormal family. We are done in this case. 

	Suppose next that $E_A((1,C^2])H^{\rm wm}= 0$. We have  $E_A((1,C^2])H=E_A((1,C^2])H^{\rm ap}$, so $H$ has an almost periodic  subrepresentation of the form:
	$$ \bigoplus_{k=1}^{d'} H(\lambda_k) \subset H$$
where $H(\lambda_k)$ is the 2-dimensional space of eigenvalue $\lambda_k\in(1, C^2]$ for all $k$, and $d'\in \N$ satisfies $2d'\geq d$. Since $H(\lambda_k)=\R^2\otimes_\R \C$, we can pick up $e_1^k,e_2^k$ which is an orthonormal basis of $\R^2$. Then  $Ie^k_j=e_j^k $ for all $j,k$ and $\{e_1^k,e_2^{k}\}_{k=1}^{d'}$ is an orthonormal family. This is the desired one.
\end{proof}

	We fix a family $\{e_i\}_{i=1}^d$ as in Lemma \ref{assumption lemma}. Put $D:=\sum_{i=1}^d \C e_i\subset H$ and $P_D$ denotes the orthogonal projection onto $D$. Observe that $e_i\in E_A([C^{-2},C^2])\subset (K_\R+ {\rm i}  K_\R) \cap I(K_\R + {\rm i} K_\R)$, so the following items make sense for all $i$:
\begin{align*}
	&W_q(e_i) = \ell_q(e_i) + \ell_q(T e_i)^*\in \Gamma_q(H_\R,U_t)\dpr;\\
	&W_q^r(e_i) = r_q(e_i) + r_q(ITI e_i)^*\in \Gamma_q(H_\R,U_t)'.
\end{align*}
where $T=I A^{-1/2}$. We have $Te_i=A^{1/2}e_i$ and $ITI e_i = A^{-1/2}e_i$ by Lemma \ref{assumption lemma}. 
Define a linear map $M\colon \mathcal F_q \to H\otimes \mathcal F_q$ by 
\begin{align*}
	M(\xi):= \sum_{i=1}^d e_i \otimes (W_q(e_i)(\xi) - W^r_q(e_i)(\xi)) .
\end{align*}
It is straightforward to check that $M\Omega =0$ and
	$$M^*M = \sum_{i=1}^d   ( W_q(e_i)-W^r_q(e_i) )^2\in \mathrm{C}^*\{ \Gamma_q(H_\R,U_t)\dpr, \Gamma_q(H_\R,U_t)'\}.$$
To prove the fullness of $\Gamma_q(H_\R,U_t)\dpr$, we have to prove that $M^*M|_{\mathcal F_q^+}$ is invertible. 
Observe that $M$ is decomposed into a sum of the following 4 maps:
\begin{align*}
	& m_\ell(\xi):= \sum_{i=1}^d e_i \otimes \ell_q(A^{1/2}e_i)^*(\xi) ,\quad  m_r(\xi):= \sum_{i=1}^d e_i \otimes r_q(A^{-1/2}e_i)^*(\xi),\\
	& m_\ell^{\dagger}(\xi):= \sum_{i=1}^d e_i \otimes \ell_q(e_i)(\xi) ,\quad  m_r^{\dagger}(\xi):= \sum_{i=1}^d e_i \otimes r_q(e_i)(\xi).
\end{align*}

\begin{lem}[{\cite[Lemma 1]{Sn03}}]
	We have $\|m_\ell\|, \|m_r\|\leq C C_1$.
\end{lem}
\begin{proof}
	It is easy to compute that
\begin{align*}
	 m_\ell^*(\xi\otimes \xi_1\otimes \cdots\otimes  \xi_n)	
	= \left(\sum_{i=1}^d \langle e_i, \xi\rangle_H A^{1/2}e_i\right)\otimes \xi_1\otimes \cdots\otimes  \xi_n
	= (A^{1/2}P_D\xi)\otimes \xi_1\otimes \cdots\otimes  \xi_n ,
\end{align*}
so that $ m_\ell^* = \Phi\circ (A^{1/2}P_D\otimes \id_{\mathcal F_q}) $. We have 
	$$\|m_\ell\| = \|m_\ell^*\| \leq \|\Phi\|\|A^{1/2}P_D\|\leq C_1 C.$$
Similarly, we have $ m_r^* = \Psi\circ (\id_{\mathcal F_q}\otimes A^{-1/2}P_D) $ and the conclusion holds.
\end{proof}

\begin{lem}[{\cite[Lemma 2]{Sn03}}]
	Put $m^\dagger:= (m^\dagger_\ell - m^\dagger_r) |_{\mathcal F_q^+}$. We have 
	$$\frac{\sqrt{\frac{d^2}{2} - (C_1C_2)^2 }}{C_2\sqrt{d}} \leq | m^\dagger|.$$
\end{lem}
\begin{proof}
	Define two linear maps $f\colon H\otimes \mathcal F_q^+ \to \mathcal F_q$ and $S\colon \mathcal F_q^+ \to \mathcal F_q^+$ by 
\begin{align*}
	& f(\xi\otimes (\eta_1\otimes \cdots\otimes \eta_n)) = \left\langle \sum_{i=1}^d e_i\otimes e_i, \xi\otimes \eta_1 \right\rangle\eta_2\otimes \cdots \otimes \eta_n ;\\
	&S(\xi_1\otimes \cdots\otimes \xi_n) = \xi_2\otimes \cdots \otimes \xi_n \otimes P_D\xi_1. 
\end{align*}
It is proved in \cite[Lemma 2]{Sn03} that $\|f\|\leq C_2 \sqrt{d}$ and $\|S\|\leq C_1C_2$. Observe that $f\circ m_\ell^\dagger = d $ and $f\circ m_r^\dagger |_{\mathcal F^+_q}= S$, hence $f\circ m^\dagger = d - S$. 
By the parallelogram identity, we have
\begin{align*}
	\frac{d^2}{2} - |S|^2
	\leq |d-S|^2 
	= |f\circ m^\dagger |^2 
	\leq  \|f^*  f\| |m^\dagger|^2.
\end{align*}
Since $\|f\|\leq C_2 \sqrt{d}$ and $\|S\|\leq C_1C_2$, we get the conclusion.
\end{proof}

\begin{proof}[Proof of Theorem \ref{full thm}]
	By Lemma \ref{assumption lemma}, we can find a family $\{e_i\}_{i=1}^d$, so that the operator $M$ is defined. 
Put $m:=(m_\ell - m_r)|_{\mathcal F^+_q} $. Then the parallelogram identity shows that 
\begin{align*}
	M^*M|_{\mathcal F^+_q}
	= |m^\dagger + m|^2 
	\geq \frac{|m^\dagger|^2}{2} - |m|^2
	\geq \frac{\frac{d^2}{2} - (C_1C_2)^2 }{2C_2^2 d} - (2C C_1)^2 >0.
\end{align*}
Hence $M^*M|_{\mathcal F^+_q}$ is invertible and we get the fullness.
\end{proof}


\begin{thebibliography}{AHHM18}

\bibitem[AH12]{AH12} H. Ando and U. Haagerup, \textit{Ultraproducts of von Neumann algebras.} J. Funct. Anal. {\bf 266} (2014), no. 12, 6842--6913.

\bibitem[AHHM18]{AHHM18} H. Ando, U. Haagerup, C. Houdayer, and A. Marrakchi, \textit{Structure of bicentralizer algebras and inclusions of type $\rm III$ factors.} To appear in Math. Ann. \texttt{arXiv:1804.05706}

\bibitem[AOS13]{AOS13}  H. Ando, I. Ojima and H. Saigo, \textit{Notes on the Krupa-Zawisza ultrapower of self-adjoint operators.} Probab. Math. Statist. {\bf 34} (2014), no. 1, 147--159. 

\bibitem[BM16]{BM16} P. Bikram and K. Mukherjee, \textit{Generator masas in $q$-deformed Araki--Woods von Neumann algebras and factoriality.} J. Funct. Anal. {\bf 273} (2017), no. 4, 1443--1478.

\bibitem[BM19]{BM19} P. Bikram and K. Mukherjee, \textit{On the commutants of generators of $q$-deformed Araki--Woods von Neumann algebras.} Preprint 2019.

\bibitem[BS90]{BS90} M. Bo$\rm\dot z$ejko and R. Speicher, \textit{An example of a generalized Brownian motion.} Comm. Math. Phys. {\bf 137} (1991), no. 3, 519--531.



\bibitem[Co74]{Co74} A.~Connes, \textit{Almost periodic states and factors of type $\rm III_1$.} J. Funct. Anal. {\bf 16} (1974), 415--445.

\bibitem[Ha85]{Ha85} U. Haagerup, \textit{Connes' bicentralizer problem and uniqueness of the injective factor of type $\rm III_1$}. Acta Math. {\bf 158} (1987), no. 1-2, 95--148.

\bibitem[Hi02]{Hi02} F. Hiai, \textit{$q$-deformed Araki--Woods algebras.} Operator algebras and mathematical physics (Constanţa, 2001), 169--202, Theta, Bucharest, 2003.

\bibitem[Ho08]{Ho08} C. Houdayer, \textit{Free Araki--Woods factors and Connes' bicentralizer problem.} Proc. Amer. Math. Soc. {\bf 137} (2009), no. 11, 3749--3755.


\bibitem[HI15]{HI15} C.~Houdayer and Y. Isono, \textit{Unique prime factorization and bicentralizer problem for a class of type $\rm III$ factors}. Adv. Math. {\bf 305} (2017), 402--455.



\bibitem[HU15]{HU15} C. Houdayer and Y. Ueda, \textit{Asymptotic structure of free product von Neumann algebras.} Math. Proc. Cambridge Philos. Soc. {\bf 161} (2016), no. 3, 489--516.



\bibitem[IM19]{IM19} Y.~Isono and A. Marrakchi, \textit{Tensor product decompositions and rigidity of full factors.} Preprint 2019. \texttt{arXiv:1905.09974}

\bibitem[KZ85]{KZ85} A. Krupa and B. Zawisza, \textit{Ultrapowers of unbounded selfadjoint operators.} Studia Math. {\bf 87} (1987), no. 2, 101--120. 

\bibitem[Ma18]{Ma18} A. Marrakchi, \textit{Full factors, bicentralizer flow and approximately inner automorphisms} Preprint 2018. \texttt{arxiv:1811.10253}


\bibitem[MT13]{MT13} { T. Masuda and R. Tomatsu}, \textit{Classification of actions of discrete Kac algebras on injective factors.} Mem. Amer. Math. Soc. {\bf 245} (2017), no. 1160, ix+118 pp.

\bibitem[Oc85]{Oc85} A. Ocneanu, \textit{Actions of discrete amenable groups on von Neumann algebras.} Lecture Notes in Mathematics, {\bf 1138}. Springer-Verlag, Berlin, 1985. iv+115 pp.

\bibitem[Po81]{Po81} S. Popa, \textit{On a problem of R.V.\ Kadison on maximal abelian $\ast$-subalgebras in factors.} Invent. Math. {\bf 65} (1981), no. 2, 269--281.

\bibitem[Ri04]{Ri04} \'E. Ricard, \textit{Factoriality of q-Gaussian von Neumann algebras.} Comm. Math. Phys. {\bf 257} (2005), no. 3, 659--665.

\bibitem[Sh97]{Sh97} D. Shlyakhtenko, \textit{Free quasi-free states.} Pacific J. Math. {\bf 177} (1997), no. 2, 329--368. 

\bibitem[SW16]{SW16} A. Skalski and S. Wang, \textit{Remarks on factoriality and q-deformations.} Proc. Amer. Math. Soc. {\bf 146} (2018), no. 9, 3813--3823.

\bibitem[Sn03]{Sn03} P. $\rm \acute{S}$niady, \textit{Factoriality of Bo$\dot{z}$ejko-Speicher von Neumann algebras.} Comm. Math. Phys. {\bf 246} (2004), no.\ 3, 561--567.



\bibitem[Ue10]{Ue10} Y.~Ueda, \textit{Factoriality, type classification and fullness for free product von Neumann algebras}, Adv.\ Math.\ {\bf 228} (2011), no.\ 5, 2647--2671. 


\bibitem[Va05]{Va05} S. Vaes, \textit{$\it \acute{E}$tats quasi-libres libres et facteurs de type $\rm III$ (d'apr$\it \grave{e}$s D.\ Shlyakhtenko).} S$\rm \acute{e}$minaire Bourbaki. Vol.\ 2003/2004. Ast$\rm \acute{e}$risque No.\ {\bf 299} (2005), Exp.\ No. 937, ix, 329--350.

\end{thebibliography}
\end{document}